\documentclass[12pt]{article}
\usepackage[centertags]{amsmath}
\usepackage{amsfonts}
\usepackage{amssymb}
\usepackage{amsthm}
\usepackage{amscd}
\usepackage{euscript}
\usepackage{color}

\theoremstyle{plain}
\newtheorem{theorem}{Theorem}[section]
\newtheorem{corollary}[theorem]{Corollary}
\newtheorem{lemma}[theorem]{Lemma}
\newtheorem{proposition}[theorem]{Proposition}
\theoremstyle{definition}
\newtheorem{definition}[theorem]{Definition}
\newtheorem{remark}[theorem]{Remark}
\newtheorem{example}[theorem]{Example}

\newcommand{\om}{\omega}
\newcommand{\Om}{\Omega}

\newcommand{\fee}{\varphi}
\renewcommand{\c}[1]{{#1}_{\mathbb{C}}}
\renewcommand{\hom}[3]{ {Hom_{\mathbb{#1}} \left ( {#2} , {#3}\right )} }
\newcommand{\tensor}[1]{{{\otimes}_{\mathbb #1}}}

\newcommand{\duzhky}[1]{{\left ( #1 \right )}}

\newcommand{\mrm}[1]{{\mathrm #1 } }

\newcommand{\ra}{\rightarrow}
\newcommand{\ssst}{\scriptscriptstyle}
\DeclareMathOperator{\dirac}{\mathcal D}
\newcommand{\Map}{\mathrm{Map}}

\newcommand{\im}{{\text{\rmfamily\upshape Im} \,}}

\newcommand{\imH}{{\im\mathbb H}}

\newcommand{\Hstar}{{\mathbb H^*}}

\DeclareMathOperator{\lspan}{\mathrm{span}}

\newcommand{\hK}{hyperK\"{a}hler }
\newcommand{\qK}{quaternionic K\"{a}hler }

\newcommand{\kahler}{{K{\" a}hler }}
\newcommand{\CRF}{{Cauchy--Riemann--Fueter\ }}

\begin{document}

\title{Nonlinear Dirac operator and quaternionic analysis.}

\author{Andriy Haydys\thanks{this paper is based upon part of the author's thesis; partially supported by the
grant "Gauge theory and exceptional geometry" (Universit\"at
Bielefeld) while preparing the final form}\\
\textit{Universit\"at Bielefeld}}
\date{May 30, 2008}
\maketitle

\begin{abstract}
Properties of the \CRF equation for maps between quaternionic manifolds are studied.
Spaces of solutions in case of maps from a K3--surface to the cotangent bundle of a
complex projective space are computed. A relationship between harmonic spinors of a
generalized nonlinear Dirac operator and solutions of the \CRF equation are
established.
\end{abstract}

\section{Introduction.}

Nonlinear generalizations of the Dirac operator were known to physicists long
ago~\cite{BaggerWitten:83} and appeared in the framework of the $\sigma$--model.
Much later they were considered by mathematicians~\cite{Taubes:99,Pidstrygach:04} in
the realm of the Seiberg--Witten theory. The basic idea of generalization is to
replace the fibre of the spinor bundle $\mathbb C^2\cong \mathbb H$, i.e. the
simplest \hK manifold, by an arbitrary \hK manifold with suitable symmetries.

On the other hand, Anselmi and Fre~\cite{AnselmiFre:94} generalized quaternionic
analy\-sis in the form of Fueter~\cite{Fueter:34,Sudbery:79} for maps between
arbitrary \hK manifolds. It is natural to expect that there is a connection between
these two approaches, since in case of the flat source manifold harmonic spinors are
exactly solutions to the \CRF equation%
\begin{equation*}
\frac{\partial u}{\partial x_0}-i\frac{\partial u}{\partial x_1}-j\frac{\partial u}{\partial x_2}-%
k\frac{\partial u}{\partial x_3}=0,\qquad x\in\mathbb H,\ u:\mathbb
H\rightarrow\mathbb H.
\end{equation*}
One of the purposes of this paper is to establish a link between the generalized
Dirac operator and quaternionic analysis.

The paper is organized as follows. In section~\ref{Sect_AholomMaps} we study
properties of the \CRF equation for maps between quaternionic manifolds. In
particular we show that its solutions are exactly those maps, whose differential has
vanishing quaternion--linear component. Therefore we call such maps aholomorphic.
Usually a \hK manifold comes equipped with some symmetries. It turns out that
certain symmetries of the target manifold force aholomorphic maps to be
(anti)holomorphic (in the usual complex sense). This in turn allows to apply a
well--developed technique of algebraic geometry to compute certain spaces of
aholomorphic maps. In the last section we show that the nonlinear Dirac operator can
be regarded as an analogue of the $\partial$--operator in complex geometry. Harmonic
spinors are shown to be twisted (in an appropriate sense) aholomorphic maps.

\section{Aholomorphic maps.}\label{Sect_AholomMaps}

\textbf{Algebraic preliminaries.} Denote by $\mathbb H$ the $\mathbb R$--algebra of
quaternions and by $Sp(1)$ the group of quaternions of unitary length. Let $(U, J_1,
J_2, J_3)$ and $(V, I_1, I_2, I_3)$ be quaternionic vector spaces. One can regard
$U$ and $V$ as natural real $Sp(1)$--representations. Further, denote by $W$ the
standard $Sp(1)$--representation given by the left multiplication on the space of
quaternions $\mathbb H$. As usual, $\mathfrak{sp}(1)$ denotes the adjoint
representation of $Sp(1)$.

\begin{proposition}
Let $\dim_{\mathbb H}U=m,\ \dim_{\mathbb H}V=n$. We have the following decomposition
into irreducible components:%
\begin{equation}\label{Eq_HomDecomposition}
\hom RUV\cong\; 4mn\,\mathfrak{sp}(1)\oplus\mathbb R^{4mn},
\end{equation}
where $\mathbb R^{4mn}$ denotes the trivial $4mn$--dimensional representation.
\end{proposition}

\begin{proof}
First observe, that a choice of quaternionic basis gives an isomorphism $U\cong
W\otimes\mathbb R^m$ and similarly $V\cong W\otimes\mathbb R^n$. Since $W^*\cong W$
we get%
\begin{equation*}
\hom RUV\cong U^*\otimes V\cong W\tensor RW\otimes\mathbb R^{mn}.
\end{equation*}
Further, since $\bar W\cong W$ we have $\c{(W\tensor RW)}\cong\c{W}\tensor
C\c{W}\cong W\otimes W\otimes\mathbb C^2\cong (S^2W\oplus\mathbb C)\otimes\mathbb
C^2\cong\bigl( \c{\mathfrak{sp}(1)} \oplus\mathbb C\bigr)\otimes\mathbb C^2$ and the
statement follows.
\end{proof}

Our next aim is to find subspaces $B_\pm\subset\hom RUV$ that give the decomposition
(\ref{Eq_HomDecomposition}).

Consider a linear map $C:\hom RUV\ra\hom RUV,$%
\begin{equation}\label{Eq_OperatorC}
C(A)=I_1AJ_1+I_2AJ_2+I_3AJ_3.
\end{equation}
A direct computation shows that $C$ satisfies the equation $C^2+2C-3=0$.
Consequently, $C$ has two eigenvalues $1$ and $-3$ and we have the decomposition
\begin{equation*}
\hom RUV=B_+\oplus B_-,
\end{equation*}
where $B_+$(resp. $B_-$) denotes the eigenspace corresponding to the eigenvalue $1$
(resp. $-3$). Observe that since $C$ is $Sp(1)$--invariant, the subspaces $B_\pm$
are also $Sp(1)$--invariant.

The trivial $Sp(1)$--subrepresentation of $\hom RUV$ is by definition the space
$\hom HUV$ of quaternion--linear maps. It is straightforward to check the inclusions
$\hom HUV\subset B_-,\, \hom HUV\otimes\mathbb R^3\subset B_+$, where the
second one is given by%
\begin{equation*}
A_1\otimes e_1+A_2\otimes e_2+ A_3\otimes e_3\mapsto I_1A_1+I_2A_2+I_3A_3.
\end{equation*}
By dimension counting we conclude that these inclusions are in fact isomorphisms.
Since $B_+$ is $Sp(1)$--invariant and complementary to the trivial
$4mn$--dimensional representation, it must be isomorphic to $4mn\,\mathfrak{sp}(1)$.
We summarize the above considerations in the following proposition.

\begin{proposition}\label{Prop_HomDecomposition}
The eigenspace $B_-$ of the linear map $C$ (\ref{Eq_OperatorC}) corresponding to the
eigenvalue $-3$ consists of quaternion--linear maps and
we have the following decomposition%
\begin{equation*}
\hom RUV\cong\hom HUV\oplus B_+.
\end{equation*}
\end{proposition}

\begin{definition}
We say that a linear map $A\in\hom RUV$ between two quaternionic vector spaces $(U,
J_1, J_2, J_3)$ and $(V, I_1, I_2, I_3)$ is \textit{aquaternionic} if
\begin{equation*}
I_1AJ_1+I_2AJ_2+I_3AJ_3=A.
\end{equation*}
\end{definition}

\begin{corollary}
An $\mathbb R$--linear map is aquaternionic if and only if its
quater\-ni\-on--linear component vanishes.
\end{corollary}


\textbf{Aholomorphic maps.} Let $(X, J_1, J_2, J_3)$ and $(M, I_1, I_2, I_3)$ be
almost hypercomplex manifolds and $u: X\ra M$ be a smooth map. Then the differential
$u_*$ is pointwise an $\mathbb R$--linear map between quaternionic vector spaces
$T.X$ and $T.M$.
\begin{definition}
We say that a map $u:X\ra M$ is \textit{aholomorphic}, if it satisfies the
\CRF equation%
\begin{equation}\label{Eq_CRF}
I_1u_*J_1 + I_2u_*J_2 + I_3u_*J_3= u_*.
\end{equation}
\end{definition}

\begin{theorem}
A map $u:X\ra M$ is aholomorphic if and only if the quaternion--linear component of
its differential vanishes at each point.\qed
\end{theorem}

Aholomorphic maps were studied under a variety of different names. They naturally
arise in the supersymmetric gauged $\sigma$--model and appeared in physical
literature \cite{AnselmiFre:94,FigueroaOFarrill:98} for the first time as
"triholomorphic maps" or "hyperinstantons". Such maps naturally appear in
higher--dimensional gauge theory~\cite{DonaldsonThomas:98} and were also studied by
Chen~\cite{Chen:99}, Chen and Li~\cite{ChenLi:00} ("quaternionic maps"),
Wang~\cite{Wang:03} ("triholomorphic maps"). Joyce~\cite{Joyce:98} ("q--holomorphic
functions") considered the case of the flat target manifold $\mathbb H$. Equation
(\ref{Eq_CRF}) was known long ago and was introduced in 1934 by Fueter
\cite{Fueter:34}("regular functions") for the simplest case of maps $u:\mathbb
H\ra\mathbb H$ in his attempts to construct a quaternionic analogue of the theory of
complex holomorphic maps.  An extensive exposition of the theory can be found in
\cite{Sudbery:79}.

In the author's opinion the proposed term \textit{aholomorphic map} better reflects
the properties of maps satisfying equation (\ref{Eq_CRF}), namely the fact that the
differential of solutions to (\ref{Eq_CRF}) has a vanishing quaternion--linear
component.

A reader can find examples of aholomorphic maps in the above mentioned sources.
Other examples will appear below.


Observe that the Cauchy--Riemann--Fuether equation (\ref{Eq_CRF}) is elliptic only
in case when a source manifold $X$ is four--dimensional. We will concentrate on this
case below.

\medskip

We now consider aholomorphic maps between \hK manifolds. A Riemannian
$4n$--dimensional manifold $M$ is called hyperK\"ahler if the holonomy group is a
subgroup of $Sp(n)$. In other words, a Riemannian manifold $(M,g)$ is \hK if it
admits three covariantly constant complex structures $I_1, I_2, I_3$ with
quaternionic relations%
\begin{equation*}
I_1I_2=-I_2I_1=I_3,\qquad I_1^2=I_2^2=I_3^2=-id,
\end{equation*}
compatible with the Riemannian structure: $g(I_l\cdot ,I_l\cdot)=g(\cdot ,\cdot),\
l=1,2,3$. Let $\om_l$ denote the \kahler 2--form corresponding to $I_l$.

\begin{proposition}\label{Prop_Weitzenboeck}
Suppose $X$ and $M$ are both \hK manifolds. If $X$ is also compact and
4--dimensional (i.e.~$X$ is a torus or a K3 surface),
then for any smooth map $u: X\ra M$ the following identity holds%
\begin{equation}\label{Eq_Weitzenboeck}
\frac 12 \|u_*\|_{L_2}^2=\frac 14 \|u_*-C(u_*)\|_{L_2}^2
-\sum\limits_{l=1}^3\int\limits_X\om_l^{\ssst X} \wedge u^*\om_l^{\ssst M}.
\end{equation}
\end{proposition}

This proposition was essentially proven by Chen and Li
\cite[Proposition~2.2]{ChenLi:00}. Formula (\ref{Eq_Weitzenboeck}) immediately
follows from the result of Chen and Li, once you observe that each \kahler form on
$X$ is self--dual and that the induced scalar product on $\Lambda^2\mathbb R^4$
is given by the sequence%
\begin{equation*}
\Lambda^2\mathbb R^4\otimes\Lambda^2\mathbb R^4\xrightarrow{\; *\,\otimes{id}\;}%
\Lambda^2\mathbb R^4\otimes\Lambda^2\mathbb R^4\xrightarrow{\;\cdot\,\wedge\cdot\ }%
\Lambda^4\mathbb R^4\cong\mathbb R.
\end{equation*}

\begin{corollary}[Vanishing theorem]\label{Cor_VanishingTheorem}
Let $X$ and $M$ be as in Proposition~\ref{Prop_Weitzenboeck}. If the cohomology
class of each \kahler form $\om_l^{\ssst M}$ on $M$ vanishes, then any aholomorphic
map $u:X\ra M$ is constant.\qed
\end{corollary}

Let $\mathcal I=\mathcal I_{\ssst M}$ denote the trivial $3$--dimensional subspace
of $\Gamma(End (TM))$ spanned by $I_1, I_2$ and $I_3$ naturally identified with
$\imH=\mathfrak{sp}(1)$.

\begin{definition}\label{Defn_PermutingAction}
An isometric action of the group $Sp(1)$ (or $SO(3)$) on a \hK manifold $M$ is
called \textit{permuting} if the subspace $\mathcal I$ is preserved and the induced
action on $\mathcal I$ is the adjoint one.
\end{definition}

The hypothesis of Corollary~\ref{Cor_VanishingTheorem} is automatically satisfied
for \hK manifolds that admit a permuting action of $Sp(1)$ or
$SO(3)$~\cite{Galicki:93}. Such actions will play a crucial role in
Section~\ref{Sect_HarmonicSpinors}. The class of \hK manifolds that admit a
permuting action is quite wide and includes a lot of interesting examples: $\mathbb
H^n$ with the flat metric and its \hK reductions with respect to the zero value of
momentum map; different moduli spaces, obtained as infinite--dimensional \hK
reductions and in particular moduli spaces of framed instantons~\cite{Maciocia:91}
over $\mathbb R^4$ and monopoles~\cite{AtiyahHitchin:88}. For any \qK manifold with
positive scalar curvature Swann~\cite{Swann:91} constructed a \hK manifold with
permuting action of $\Hstar\supset Sp(1)$.


\medskip

Let us consider a lager class of target \hK manifolds. Namely suppose that %
$(M, I_1, I_2, I_3)$ admits only an action of $S^1$
that fixes one complex structure, say $I_1$, and rotates the other two, i.e.%
\begin{equation}\label{Eq_PermutingS1Action}
\duzhky{L_z}_*\, I_1= I_1\duzhky{L_z}_*,\qquad %
\duzhky{L_z}_*\, I_w\duzhky{L_{\bar z}}_*= I_{zw},
\end{equation}
where $L_z:M\ra M$ denotes the left shift by $z\in S^1$, $w=a+bi$ is a complex
number of unitary length, $I_w=aI_2+bI_3$.

\begin{proposition}\label{Prop_AholomorphicMapsAreAntiholom}
Let $X$ be a compact \hK 4--dimensional manifold. Assume that the target \hK
manifold $M$ admits an isometric action of $S^1$ that fixes $I_1$ while rotating
$I_2$ and $I_3$. Then a map $u: X\ra M$ is aholomorphic if and only if it is $(J_1,
I_1)$--antiholomorphic.
\end{proposition}

\begin{proof}
First notice that existence of an isometric $S^1$--action that fixes $I_1$ and
rotates the other two complex structures implies that $\om_2^{\ssst M}$ and
$\om_3^{\ssst M}$ are exact~\cite{Hitchin:87}. Indeed, denote by $K$ the Killing
vector field of the $S^1$--action and by $\mathcal L_{K}$ the Lie derivative.
It follows from (\ref{Eq_PermutingS1Action}) that %
$\mathcal L_{K}\,\om_2^{\ssst M}=\om_3^{\ssst M}$. Applying the Cartan formula,
one gets %
$\om_3^{\ssst M}=\mathcal L_{K}\,\om_2^{\ssst M}=d(\imath_{K}\om_2^{\ssst M})$. %
Therefore equality~(\ref{Eq_Weitzenboeck}) takes the following form:%
\begin{equation}\label{Eq_WeitzenboeckSimplified}
\frac 12\| u_* \|_{L_2}^2=\frac 14\| u_*-C(u_*)\|_{L_2}^2 -\int\limits_X
\om_1^{\ssst X}\wedge u^*\om_1^{\ssst M}.
\end{equation}

Further, since the action of $S^1$ is isometric, the energy functional%
\begin{equation}\label{Eq_EnergyFunctional}
E(u)=\frac 12 \| u_* \|_{L_2}^2=\frac 12\int\limits_X\| u_*\|^2\, dvol_X,\qquad
u:X\ra M
\end{equation}
is $S^1$--invariant.

Suppose now that $u$ is aholomorphic. In particular, $u$ is an absolute minimum of
the energy functional~(\ref{Eq_EnergyFunctional}) within its homotopy class
$\alpha=[u]$. Then for each $z\in S^1$ the map $u^z=L_z\circ u$ lies in the same
homotopy class $\alpha$ and is also an absolute minimum of the energy functional
within $\alpha$. From equation~(\ref{Eq_WeitzenboeckSimplified}) we conclude that
$u^z$ must be aholomorphic:%
\begin{equation*}
I_1u_*^zJ_1 + I_2u_*^zJ_2 + I_3u_*^zJ_3= u_*^z.
\end{equation*}
The above equation can be rewritten as%
\begin{equation*}
I_1u_*J_1 + \duzhky{L_{\bar z}}_*I_2\duzhky{L_{z}}_*\,u_*J_2 + %
\duzhky{L_{\bar z}}_*I_3\duzhky{L_{z}}_*\,u_*J_3= u_*.
\end{equation*}
In particular, for $z=-1$ we get $I_1u_*J_1 - I_2u_*J_2 - I_3u_*J_3= u_*$. Since $u$
satisfies \CRF equation~(\ref{Eq_CRF}), we obtain $I_1u_*J_1= u_*$, i.e. $u$ is
$(J_1, I_1)$--antiholomorphic. On the other hand, it easily follows from the
definition that any antiholomorphic map is also aholomorphic.
\end{proof}

An example of the circle action preserving one complex structure and rotating the
other two is the standard fiberwise action on the cotangent bundle %
$\mathrm T^*\mathbb P^n$ of the complex projective space $\mathbb P^n$ equipped with
the Calabi metric. More generally, Kaledin~\cite{Kaledin:97} and independently
Feix~\cite{Feix:01} constructed such metrics on a neighborhood of the zero section
in $T^*Z$ for real--analytic \kahler manifolds $Z$.

\begin{example}[Aholomorphic maps from a K3--surface into $\mathrm T^*\mathbb P^n$]\footnote{I am
grateful to I.Panin for helpful discussions on this example}%
It follows from Proposition~\ref{Prop_AholomorphicMapsAreAntiholom} that any
aholomorphic map $u: X\rightarrow \mathrm T^*\mathbb P^n$ must be $(J_1,
I_1)$--antiholomorphic. The standard antiholomorphic automorphism of $\mathbb P^n$
induces an $I_1$--antiholomorphic automorphism on $\mathrm T^*\mathbb P^n$ and
therefore we have a natural bijection between the spaces of holomorphic and
antiholomorphic maps into $\mathrm T^*\mathbb P^n$. Further, each holomorphic map
$u: X\rightarrow \mathrm T^*\mathbb P^n$ naturally decomposes into the projection
$\fee: X\rightarrow \mathbb P^n$ and a holomorphic section %
$s\in\Gamma(\fee^*\mathrm T^*\mathbb P^n)$. Each nonconstant holomorphic map $\fee$
into the projective space can be obtained as a morphism associated to a linear
system of a positive base point free divisor $D$ (see~\cite{GriffithsHarris:78} for
example). Then $\fee^*\mathcal O_{\mathbb P^n}(1)=\mathcal O_X(D)$ and the
pull--back to $X$ of the short exact sequence dual
to the Euler one has the following form:%
\begin{equation*}
0\rightarrow \fee^*\mathrm T^*\mathbb P^n\rightarrow (n+1)\mathcal O_X(-D)
\rightarrow \mathcal O_X\rightarrow 0.
\end{equation*}
It follows that the lift $s$ must be the zero section. Therefore our problem reduces
to description of the space of holomorphic maps $X\rightarrow\mathbb P^n$.

From now on we assume that $X$ is a K3--surface. Let a map $\fee=\fee_D:X\rightarrow
\mathbb P^n$ be given by the complete linear system $|D|=\mathbb P(H^0(\mathcal
O_X(D)))$ with empty base locus of a positive divisor $D$. Since $D^2$ is the degree
of $\fee$ we may assume $D^2\ge 0$. It turns out that base point free divisors $D$ on a
K3--surface admit a purely numeric characterization at least if $D$ is big ($D^2>0$).
Indeed, for a big divisor $D$ the associated morphism $\fee: X\rightarrow \mathbb P^n$ is
generically finite. It follows that $\dim \fee (X)=2$ and by Bertini's
theorem~\cite{BarthHulek:04} the divisor $D$ is linearly equivalent to a smooth
irreducible curve $C$ with $p_a(C)=1+ C^2/2=1+D^2/2>1$. It follows that $D$ is nef (numerically
effective) and by assumption $D$ is also big ($D^2>0$). On the other hand, the complete linear
system $|D|$
of a nef and big divisor $D$ has a non--empty base locus if and only if $D=kE+R$, where
$R$ is a smooth rational curve ($R^2=-2$), $E$ is a smooth elliptic curve ($E^2=0$),
$E\cdot R=1$ and $k\ge 2$~\cite{Friedman:98}. In case when $D$ is big, nef and base point
free the associated morphism is a map into the complex projective space of dimension
$N=1+D^2/2$. 

Now fix a positive integer $n$. Let $\fee :X\rightarrow\mathbb P^n$ be a holomorphic
map and $k$ be the dimension of the projective span of $\fee (X)$, that is $\fee$
decomposes as $i\circ\psi$, where $i:\mathbb P^k\hookrightarrow\mathbb P^n$ is a
standard embedding. Then $\psi$ is a morphism given by a projective basis of a
linear $k$--dimensional system of $D$ with $k\le N=\dim |D|=1+ D^2/2$. One can
describe $\psi$ equivalently as the composition of a morphism $\fee_D:
X\rightarrow\mathbb P^N$, given by a projective basis of the complete linear system
$|D|$, and a projection %
$\pi:\mathbb P^N\setminus\mathbb P^{N-k-1}\rightarrow\mathbb P^k$ with respect to a
subspace $\mathbb P^{N-k-1}\hookrightarrow\mathbb P^N$ that does not intersect the
image of $\fee_D$. Denote by $\mathbb P\mathrm V_{N-k}(\mathbb C^{N+1})$ the
projective Stiefel manifold, %
$\mathbb P\mathrm V_{N-k}(\mathbb C^{N+1})\xrightarrow{\, pr\,} Gr_{N-k}(\mathbb
C^{N+1})$. Then the space $\Psi (D,k)$ of all holomorphic maps $\psi:
X\rightarrow\mathbb P^k$ that can be obtained by a choice of $k+1$ linearly
independent sections of $\mathcal O_X(D)$ without base locus is the Zariski--open set%
\begin{equation}\label{Eq_PsiDk}
pr^{-1}\{\, V\in Gr_{N-k}(\mathbb C^{N+1})\ |\ [V]\cap \fee_D(X)=\emptyset\, \}\subset%
\mathbb P\mathrm V_{N-k}(\mathbb C^{N+1}).
\end{equation}
In particular, $\Psi (D, k)$ is connected.

Observe also that $\mathbb P^N\setminus\mathbb P^{N-k-1}$ is the total space of a
vector bundle over $\mathbb P^k\hookrightarrow\mathbb P^N$. Therefore
$\pi\circ\fee_D$ is homotopic to $\fee_D$. It follows that the spaces $\Psi(D,k)$
and $\Psi(D',k')$ lie in the same component of $\Map(X,\mathbb P^n)$ if and only if
$D=D'$.

Summing up the above considerations we get the following result.

\begin{proposition}
Let $X$ be a K3--surface and $n$ be a positive integer. Then the space of all
nonconstant aholomorphic maps $u: X\rightarrow \mathrm T^*\mathbb P^n$
(equivalently, the space of all holomorphic maps  $\fee: X\rightarrow \mathbb P^n$)
is the stratified space%
\begin{equation*}
\bigsqcup_{D,\, k} \Psi(D,k)\times Gr_{k+1}(\mathbb C^{n+1}),
\end{equation*}
where $D$ denotes a divisor class on $X$ such that the base locus of the complete
linear system $|D|$ is empty; $k$ is an integer, $1\le k\le n$; $\Psi(D,k)$ is a
Zariski--open subset (\ref{Eq_PsiDk}) of the projective Stiefel manifold $\mathbb
P\mathrm V_{N-k}(\mathbb C^{N+1}), N=1+D^2/2$.

The subsets $\Psi(D,k)\times Gr_{k+1}(\mathbb C^{n+1})$ and %
$\Psi(D',k')\times Gr_{k+1}(\mathbb C^{n+1})$ lie in the same component of $\Map (X,
\mathrm T^*\mathbb P^n)$ if and only if $D=D'$.\qed
\end{proposition}

Observe that any holomorphic map $X\rightarrow \mathbb P^n$ can be obtained from a
morphism $\fee_D$, given by a primitive divisor $D$, composing it with holomorphic
maps between complex projective spaces. In the examples below we will indicate only
the primitive divisor classes.

For a K3--surface $X$ the Euclidean lattice $(H^2(X;\mathbb Z), \cup)$ is isomorphic
to $-2E_8\oplus 3H$ and is of rank 22. Notice that by the Torelli theorem a
K3--surface can be specified by its period point (see~\cite{BarthHulek:04} for
details). The natural map $c_1:Pic(X)\rightarrow H^2(X,\mathbb Z)$ is injective and
the rank $\rho$ of its image, the Neron--Severi group $NS(X)$, can take any integer
value between $0$ and $20$. We shall consider some examples for small values of
$\rho$ (the examples are taken from~\cite{Belcastro:02}).

$1)\; \rho=1, NS(X)\cong \langle 2k^2 \rangle$. We have just one primitive divisor
class $D$ with $D^2=2k^2$. The complete linear system is base point free and we get
a regular holomorphic map $\fee_D: X\rightarrow\mathbb P^{1+k^2}$.

$2)\;\rho=2, NS(X)\cong H$. In this case the two generators $E_1$ and $E_2$ are
elliptic curves. The complete linear system $|E_i|$ is base point
free~\cite{Friedman:98} and defines a regular map $\fee_{E_i}:X\rightarrow \mathbb
P^1$. Any other positive primitive divisor class has the form $D=pE_1+qE_2$, where
$p$ and $q$ are positive coprime integers. The base locus of $D$ is empty and
therefore we get a map $\fee_D:X\rightarrow\mathbb P^{2pq}$.

$3)\;\rho=2, NS(X)\cong
  \begin{pmatrix}
    2 & \phantom{-}1 \\
    1 & -2
  \end{pmatrix}
$. Let $C$ and $R$ be generators, $C^2=2, R^2=-2, C\cdot R=2$. It is easy to check
that there are no divisors with vanishing self--intersection number. A divisor
$D=pC+qR$ is nef iff $p\ge q$. In this case $D$ is also big. Consequently, for %
coprime $p$ and $q,\ p\ge q$ the divisor $D$ is primitive with empty base locus and
we get a holomorphic map $\fee_D: X\rightarrow \mathbb P^{1+p^2+2pq-q^2}$.
\end{example}


\section{Harmonic spinors as twisted aholomorphic maps.}\label{Sect_HarmonicSpinors}

In this section we establish connection between aholomorphic maps and (generalized)
harmonic spinors. Before proceeding let us briefly outline the definition of the
(generalized) Dirac operator in a form suitable for our purposes. Details can be
found in the paper of Pidstrygach~\cite{Pidstrygach:04}, where the Dirac operator is
defined in a slightly more general context.

\textbf{Algebraic preliminaries.} Recall that the group $Spin(4)$ is isomorphic to
the product of two copies of $Sp(1)$:
\begin{equation*}
Spin(4)=Sp_+(1)\times Sp_-(1),
\end{equation*}
where we use subscripts $"\pm"$ to distinguish between different copies. The
isomorphism $Spin(4)/\pm 1\cong SO(4)$ is given by%
\begin{equation}\label{Eq_Spin4SO4}
(q_+, q_-)\mapsto B_{q_+,q_-}:\mathbb H\ra \mathbb H,\quad B_{q_+,q_-}h=q_-h\bar
q_+.
\end{equation}
Recall that $W$ denotes the $Sp(1)$--representation on $\mathbb H$ by multiplication
on the left. Further, $\mathbb R^4$ denotes the standard $SO(4)$--representation.
The following homomorphism of $Spin(4)=Sp(1)_+\times Sp_-(1)$--representations%
\begin{equation*}
\mathbb R^4\otimes W^+\ra W^-, \qquad h_1\otimes h_2\mapsto h_1h_2,
\end{equation*}
is called the Clifford multiplication, where $W^\pm$ denote the representations
induced by $Sp_\pm(1)$.

\medskip

Let $(V, I_1, I_2, I_3)$ be a quaternionic vector space. We identify the subspace
$\lspan (I_1, I_2, I_3)\subset End_{\mathbb R}(V)$  with $\imH=\mathfrak{sp}(1)$.

\begin{lemma}\label{Prop_WasFactorInRepresentation}
Let $V$ be a real representation of the group $Sp(1)$ such that the induced action
on $End_{\mathbb R}(V)$ preserves $\lspan(I_1, I_2, I_3)$. Assume that the induced
representation on $\lspan(I_1, I_2, I_3)$ coincides with the adjoint one. Then%
\begin{equation}\label{Eq_WasFactorInRepresentation}
\c V\cong W\otimes E,
\end{equation}
where $E$ is an $Sp(1)$--representation.
\end{lemma}

The proof is a straightforward calculation of the actions of $i,j,k\in Sp(1)$ on %
$\mathbb C\tensor R V\cong\mathbb C^2\tensor CV_{I_1}$, where $V_{I_1}$ denotes the
complex vector space $(V, I_1)$.

From (\ref{Eq_WasFactorInRepresentation}) one obtains a variant of the Clifford
multiplication:%
\begin{equation}\label{Eq_ComplCliffordMultiplication}
\c{\mathbb R}^4\otimes\c V\cong \c{\mathbb R}^4\otimes W^+\otimes E\ra W^-\otimes E.
\end{equation}
Observe that the homomorphism (\ref{Eq_ComplCliffordMultiplication}) induces an
homomorphism between real parts:%
\begin{equation*}
\mathbb R^4\otimes V\ra [W^-\otimes E]_r.
\end{equation*}

\textbf{The nonlinear Dirac operator.} From now on $X^4$ is a smooth closed oriented
Riemannian manifold. We also assume that $X$ is a spin manifold with a
$Spin(4)$--principal bundle $\pi :P\ra X$. The assumption that $X$ is spin is not
essential, since the bundle $P$ can be replaced by a $Spin^c(4)$--bundle,  however
the exposition becomes a bit clearer in case of spin manifolds.

Let $(M, I_1, I_2, I_3)$ be a \hK manifold with permuting action of $Sp(1)$ (see
Definition~\ref{Defn_PermutingAction}). Then the group $Spin(4)$ acts on $M$ via the
homomorphism $Spin(4)\ra Sp_+(1)$ and we get the associated fibre bundle:%
\begin{equation*}
\mathbb{M}=P\times_{Spin(4)} M.
\end{equation*}
Observe that in case $M=\mathbb H$ with the standard action of $Sp(1)$ one gets the
usual positive spinor bundle $\mathcal W^+$. Therefore sections of $\mathbb M$ are
called (generalized) spinors. Notice that the space of spinors can be naturally
identified
with the space of equivariant maps:%
\begin{equation*}
\Gamma(\mathbb M)\cong \Map^{Spin(4)}(P, M).
\end{equation*}
The Levi--Civita connection on $P$ determines the covariant derivative:%
\begin{equation*}
\nabla_{\mrm v}\,u=u_*(\hat{\mrm v}),\qquad u\in\Map^{Spin(4)}(P, M)
\end{equation*}
where $\hat{\mrm v}$ denotes the horizontal lift of $\mrm v\in TX$.
In other words we get a map%
\begin{equation*}
\nabla: \Map^{Spin(4)}(P, M)\ra\Gamma(T^*X\otimes \pi_!(u^*TM)),
\end{equation*}
where $\pi_!(u^*TM)\ra X$ denotes the factor of $u^*TM\ra P$ by the group action.

\begin{remark}\label{Rem_CovDerivativeAsSection}
Strictly speaking, the operator $\nabla$ is not well--defined, since its range
depends on the element of the domain. However one can define $\nabla$ as a section
of a certain vector bundle as follows. Consider the evaluation map $ev:
\Map^{Spin(4)}(P,M)\times P\ra M,\ (u,p)\mapsto u(p)$.
Then one gets the following diagram%
\begin{displaymath}
 \begin{CD}
 ev^*TM  @> > > TM  \\
 @VVV          @VVV\\
 \Map^{Spin(4)}(P,M)\times P @>ev > > M.
 \end{CD}
\end{displaymath}
Dividing the first column by the group action one gets a vector bundle %
$\mathcal E$ over an the infinite dimensional space
$\Map^{Spin(4)}(P,M)\times X$; the restriction of $\mathcal E$ to %
$\{ u \}\times X$ coincides with $\pi_!(u^*TM)$. Then $\nabla$ is well--defined as a
section of $\mathcal E\ra \Map^{Spin(4)}(P,M)\times X$.

In order to keep the exposition clear, we will not keep to the above formalism of
vector bundles over infinite--dimensional spaces.
\end{remark}

The tangent space of $M$ has a natural structure of quaternionic vector space at
each point. Since we have a permuting action of $Sp(1)$ on $M$, we get from the
Proposition~\ref{Prop_WasFactorInRepresentation} that $\c TM\cong W^+\otimes\tilde
E$, where $\tilde E$ is a complex vector bundle over $M$ with an action of $Sp(1)$
($\tilde E$ coicides with $(TM, I_1)$ as a complex vector bundle, however the action
of $Sp(1)$ is different from the induced one). Consequently, we
get the Clifford multiplication%
\begin{equation}\label{Eq_CliffordMultiplicationOnBundles}
Cl: T^*X\otimes\pi_!(u^*TM)\ra [\mathcal W^-\otimes E]_r,
\end{equation}
where $E=\pi_!(u^*\tilde E)$ and $\mathcal W^-$ is the negative spinor bundle over
$X$.

\begin{definition}[\cite{Pidstrygach:04}]\label{Defn_DiracOperator}
The first order differential operator $\dirac$ defined by the sequence%
\begin{equation*}\label{Eq_DiracOperator}
\dirac : \Gamma\duzhky{\mathbb M}\xrightarrow{\ \nabla\ }%
\Gamma\bigl (T^*X\otimes\pi_!(u^*TM)\bigr )\xrightarrow{\, Cl\ }%
\Gamma \bigl ([ \mathcal W^-\otimes E]_r \bigr )
\end{equation*}
is called a (generalized) \textit{Dirac operator}.
\end{definition}

\begin{definition}
A spinor $u$ such that $\dirac u=0$ is called \textit{harmonic}.
\end{definition}

\begin{remark}
The Dirac operator is well--defined as a section of a vector bundle over an
infinite--dimensional space similarly as the covariant derivative (see
Remark~\ref{Rem_CovDerivativeAsSection}) and it is a Fredholm
section~\cite{Pidstrygach:04}. In case when the fibre $M$ of the spinor bundle is a
copy of quaternions with the standard action of $Sp(1)$ one recovers the usual
linear Dirac operator.
\end{remark}

\begin{remark}
Notice that if the target \hK manifold admits a permuting action of $SO(3)$ rather
then $Sp(1)$ one needs just the principal $SO(4)$--bundle of orthonormal frames
rather then its $Spin(4)$--lifting to define the Dirac operator. A well--known
example in classical theory is $d^+ +d^*: \Om^1(X)\ra\Om^2_+(X)\oplus\Om^0(X)$.

One can also define a Dirac operator with the help of a $Spin^c(4)$--structure. In
this case the target manifold $M$ is required to carry a triholomorphic action of
$S^1$ commuting with the permuting action of $Sp(1)$.
\end{remark}

The Levi--Civita connection splits $TP$ into horizontal and vertical bundles:
$TP\cong \mathcal H\oplus\mathcal V$. Since we have a natural projection $pr: P\ra
P_{\ssst{SO}}$ onto the principal bundle of orthonormal frames of $X$, the
horizontal bundle $\mathcal H$ has a natural quaternionic structure  $(J_1, J_2,
J_3)$, which is defined as follows: a point $p\in P$ determines an orthonormal basis
$\mrm v=pr (p)=(\mrm v_0, \mrm v_1, \mrm v_2, \mrm v_3)$ of $T_{\pi (p)}X$ and
consequently a basis $\hat{\mrm v}$ of $\mathcal H_p$; then $(J_1, J_2, J_3)$ is
defined as the unique quaternionic structure\footnote{the basis $\mrm v$ can be
viewed as an isomorphism $T_xX\cong\mathbb H$; then the quaternionic structure on
$\mathbb H$, compatible with the isomorphism~(\ref{Eq_Spin4SO4}),  is the
\textit{right} one} such that $\hat{\mrm v}_l= -J_l\hat{\mrm v}_0$ for $l=1,2,3$.

\begin{theorem}
For a map $u\in\Map^{Spin(4)}(P, M)$ denote by $u_*^h$ the restriction of the
differential $u_*$ to the horizontal subbundle $\mathcal H\subset TP$. Then the
spinor $u$ is harmonic if and only if the Cauchy--Riemann--Fueter--type equation holds:%
\begin{equation*}
I_1u_*^hJ_1 + I_2u_*^hJ_2+I_3u_*^hJ_3=u_*^h.
\end{equation*}
\end{theorem}

\begin{proof}
Pick a point $p\in P$ and denote by $\mrm v=pr(p)$ the basis of $TX$ at the point
$x=\pi(p)$ as before. Let $u\in\Map^{Spin(4)}(P, M)$ be a harmonic spinor,
$m=u(p)\in M$.

The pull--back of the Clifford
multiplication~(\ref{Eq_CliffordMultiplicationOnBundles}) to $P$ can be described by
the following sequence
(see~(\ref{Eq_WasFactorInRepresentation}),(\ref{Eq_ComplCliffordMultiplication})):%
\begin{equation}\label{Eq_CliffordMultiplSequence}
\begin{aligned}
\mathbb R^4\otimes u^*TM \hookrightarrow &\c{\duzhky{\mathbb R^4\otimes u^*TM}}\cong%
\c{\mathbb R}^4\otimes u^*\c TM\\%
\cong & (W^-\otimes W^+)\otimes (W^+\otimes TM) \ra%
W^-\otimes TM.
\end{aligned}
\end{equation}
Notice that the image of the above map lies automatically in the real part of %
$\, W^-\otimes TM$. Further,  $\mathbb R^4$ and $W^\pm$ are isomorphic to $\mathbb
H$ as vector spaces. Then the
homomorphism~(\ref{Eq_CliffordMultiplSequence}) is the following map:%
\begin{equation*}
h\otimes w\mapsto h\cdot 1\otimes w - h\cdot j\otimes I_2w,\qquad h\in\mathbb H,\
w\in TM.
\end{equation*}

Observe that $h=1$\,(resp. $i,j,k$) corresponds to the horizontal lift of $\mrm
v_0$\,(resp. $\mrm v_1,\mrm v_2,\mrm v_3$). According to the definition of the Dirac
operator take %
$h=1, w=\nabla_{\hat{\mrm v}_0}\, u=u_*(\hat{\mrm v}_0)=u_*^h (\hat{\mrm v}_0);\ %
h=i, w=u_*^h (\hat{\mrm v}_1)\dots$ %
and sum up the result. It follows that $u$ is harmonic at
the point $x$ if and only if the following equation holds%
\begin{align*}
&\duzhky{1\otimes u_*^h (\hat{\mrm v}_0)-j\otimes I_2 u_*^h (\hat{\mrm v}_0)}+%
\duzhky{i\otimes  u_*^h (\hat{\mrm v}_1)-k\otimes I_2u_*^h (\hat{\mrm v}_1)}+\\%
+&\duzhky{j\otimes u_*^h (\hat{\mrm v}_2) + 1\otimes I_2 u_*^h (\hat{\mrm v}_2)}+%
\duzhky{k\otimes u_*^h (\hat{\mrm v}_3)+i\otimes I_2 u_*^h (\hat{\mrm v}_3)}=0.
\end{align*}
After a simplification one gets%
\begin{equation}\label{Eq_DopHarmonicity}
\begin{aligned}
&1\otimes \left ( u_*^h (\hat{\mrm v}_0)+I_1u_*^h (\hat{\mrm v}_1)+I_2u_*^h (\hat{\mrm
v}_2)+I_3u_*^h (\hat{\mrm v}_3)  \right )\\
+& j\otimes \left ( -I_2u_*^h (\hat{\mrm v}_0)+I_3u_*^h (\hat{\mrm v}_1)+u_*^h(\hat{\mrm
v}_2)-I_1u_*^h (\hat{\mrm v}_3)  \right )=0.%
\end{aligned}
\end{equation}
It is easy to see that~(\ref{Eq_DopHarmonicity}) is equivalent to the single
equation $u_*^h (\hat{\mrm v}_0)+I_1u_*^h (\hat{\mrm v}_1)+I_2u_*^h (\hat{\mrm
v}_2)+I_3u_*^h (\hat{\mrm v}_3)=0$. Recalling that %
$\hat{\mrm v}_l=-J_l\hat{\mrm v}_0,\ l=1,2,3$ we get%
\begin{equation*}
\tilde C(u_*^h)\,\hat{\mrm v}_0=u_*^h \hat{\mrm v}_0-I_1u_*^h J_1 \hat{\mrm
v}_1-I_2u_*^h J_2 \hat{\mrm v}_2-I_3u_*^h J_3 \hat{\mrm v}_3=0.
\end{equation*}
Observe that for any fixed $l=1,2,3$ the equations $\tilde C(u_*^h)\,\hat{\mrm
v}_0=0$ and $\tilde C(u_*^h)\,J_l\hat{\mrm v}_0=0$ are equivalent. It remains to
note that $\hat{\mrm v}_0, J_1\hat{\mrm v}_0, J_2\hat{\mrm v}_0$ and $J_3\hat{\mrm
v}_0$ span the horizontal subspace at the point $p\in P$ and therefore %
$\tilde C(u_*^h)=0$.
\end{proof}

\begin{corollary}
A spinor $u$ is harmonic if and only if the horizontal part $u_*^h$ of its
differential has no quaternion--linear component.\qed
\end{corollary}

%
%

The above Corollary reveals a deep analogy between the Dirac operator and the
$\partial$--operator of complex geometry. Indeed, let $(Y, I_{\ssst Y})$ and $(Z,
I_{\ssst Z})$ be complex manifolds. Assume that a Lie group $G$ acts holomorphically
on $Z$ and pick a $G$--principal bundle $\pi^{\ssst G}: P_{\ssst G}\ra Y$ with a
connection.
Then the operator $\partial$ can be defined in the usual way, namely%
\begin{equation*}
\partial : \Gamma(P_{\ssst G}\times_G Z)\xrightarrow{\ \nabla\ }%
    \Om^1(Y)\otimes\Gamma\bigl ( \pi_!^{\ssst G}(u^*TZ)\bigr )\ra%
    \Om^{1,0}(Y)\otimes\Gamma\bigl ( \pi_!^{\ssst G}(u^*TZ)\bigr ).
\end{equation*}

Further, we also have a splitting of the tangent bundle of $P_{\ssst G}$ into
vertical and horizontal parts: $TP_{\ssst G}=\EuScript V\oplus\EuScript H$. Observe
that the bundle $\EuScript H$ inherits a complex structure from $TY$.
Then an equivariant map $u$ satisfies the equation $\partial u=0$ if and only if %
$I_{\ssst Z}u_*^h I_{\ssst Y}=u_*^h$ or, in other words, if and only if the
horizontal component has no complex--linear component.
\bigskip

\bibliographystyle{acm}

\end{document}